\documentclass[12pt]{amsart}
\usepackage{amssymb,amsmath,amsfonts}              
\usepackage{mathrsfs}
\usepackage{graphicx}
\usepackage{hyperref}
\usepackage{amsthm}
\usepackage{multicol}
\usepackage{color}
\usepackage{comment}
\usepackage{float,wrapfig}
\usepackage{lipsum}
\usepackage{epstopdf}
\usepackage{enumerate}
\renewcommand{\epsilon}{\varepsilon}
\DeclareMathOperator{\dvg}{div}

\DeclareMathOperator{\dist}{dist}

\DeclareMathOperator{\graph}{graph}

\DeclareMathOperator{\tr}{tr}
\DeclareMathOperator{\diam}{diam}

\def\pd{\partial}
\def\cd{\nabla}
\def\R{\mathbb{R}}

\def\N{\mathbb{N}}

\def\M{\mathcal{M}}

\newcommand{\inner}[2]{\left\langle#1,#2\right\rangle} 

\def\ba #1\ea {\begin{align} #1\end{align}}
\def\bann #1\eann {\begin{align*} #1\end{align*}}
\def\ben #1\een {\begin{enumerate} #1\end{enumerate}}
\def\bi #1\ei {\begin{itemize}\renewcommand\labelitemi{--} #1\end{itemize}}

\newtheorem{theorem}{Theorem}[section]
\newtheorem{lemma}[theorem]{Lemma}

\newtheorem{remark}[theorem]{Remark}

\newtheorem{corollary}[theorem]{Corollary}
\newtheorem{definition}[theorem]{Definition}

\newtheorem{claim}[theorem]{Claim}

\newtheoremstyle{TheoremNum}
        {\topsep}{\topsep}              
        {\itshape}                      
        {}                              
        {\bfseries}                     
        {.}                             
        { }                             
        {\thmname{#1}\thmnote{ \bfseries #3}}
    \theoremstyle{TheoremNum}

\title[Convex ancient solutions to MCF]{Convex ancient solutions to mean curvature flow}

\author{Theodora Bourni}
\author{Mat Langford}
\author{Giuseppe Tinaglia}
\thanks{The third author was partially supported by EPSRC grant no. EP/M024512/1}

\address{Department of Mathematics, University of Tennessee Knoxville, Knoxville TN, 37996-1320}
\email{tbourni@utk.edu}
\email{mlangford@utk.edu}
\address{Department of Mathematics, King's College London, London, WC2R 2LS, U.K.}
\email{giuseppe.tinaglia@kcl.ac.uk}

\date{\today}

\begin{document}

\begin{abstract}
X.-J.~Wang \cite{Wa11} proved a series of remarkable results on the structure of (non-compact) convex ancient solutions to mean curvature flow. Some of his results do not appear to be widely known, however, possibly due to the technical nature of his arguments and his exploitation of methods which are not widely used in mean curvature flow. In this expository article, we present Wang's structure theory and some of its consequences. We shall simplify some of Wang's analysis by making use of the monotonicity formula and the differential Harnack inequality, and obtain an important additional structure result by exploiting the latter. We conclude by showing that various rigidity results for convex ancient solutions and convex translators follow quite directly from the structure theory, including the new result of Corollary \ref{cor:bowl uniqueness}. We recently provided a complete classification of convex ancient solutions to curve shortening flow by exploiting similar arguments \cite{BLT3}.
\end{abstract}

\maketitle

\tableofcontents

\section{Introduction}

\thispagestyle{empty}

A smooth one-parameter family $\{\M^n_t\}_{t\in I}$ of smoothly immersed hypersurfaces $\M_t^n$ of $\R^{n+1}$ \emph{evolves by mean curvature flow} if there exists a smooth one-parameter family $X:M^n\times I\to\R^{n+1}$ of immersions $X(\cdot,t):M^n\to \R^{n+1}$ with $\M^n_t=X(M^n,t)$ satisfying
\[
\pd_tX(x,t)=\vec H(x,t)\quad\text{for all}\quad (x,t)\in M^n\times I\,,
\]
where $\vec H(\cdot,t)$ is the mean curvature vector field of $X(\cdot,t)$. 
Unless otherwise stated, we shall \emph{not} assume that $M^n$ is compact. 

We will refer to a hypersurface as (\emph{strictly}) \emph{convex} if it bounds a (strictly) convex body, \emph{locally uniformly convex} if its second fundamental form is positive definite and (\emph{strictly}) \emph{mean convex} if its mean curvature (with respect to a consistent choice of unit normal field) is  (positive) non-negative. We will refer to a mean curvature flow $\{\M^n_t\}_{t\in I}$ as compact/(strictly) convex/locally uniformly convex/mean convex if each of its timeslices $\M^n_t$ possesses the corresponding property.

An \emph{ancient} solution to mean curvature flow is one which is defined on a time interval of the form $I=(-\infty,T)$, where $T\leq \infty$. Unless otherwise stated. Ancient solutions are of interest due to their natural role in the study of high curvature regions of the flow since they arise as limits of rescalings about singularities \cite{HamiltonPinched,Hu90,Hu93,HuSi99b,HuSi99a,Wh00,Wh03}.

A special class of ancient solutions are the \emph{translating solutions}. As the name suggests, these are solutions $\{\M^n_{t}\}_{t\in(-\infty,\infty)}$ which evolve by translation: $\M^n_{t+s}=\M^n_t+se$ for some fixed vector $e\in \R^{n+1}$. The timeslices $\M^n_t$ of a translating solution $\{\M^n_{t}\}_{t\in(-\infty,\infty)}$ are all congruent and satisfy the \emph{translator equation}, 
\[
\vec H=e^\perp\,,
\]
where $\cdot^\perp$ denotes projection onto the normal bundle. 
Translating solutions arise as blow-up limits at \emph{type-II} singularities, which are still not completely understood \cite{HamiltonPinched,HuSi99b,HuSi99a} (although there is now a classification of mean convex examples in\footnote{The classification of all proper translating curves in the plane is left as an exercise.} $\R^3$ \cite{AlWu,BLT2,HIMW,SX,Wa11}). Understanding ancient and translating solutions is therefore of relevance to applications of the flow which require a controlled continuation of the flow through singularities. 

Ancient solutions to mean curvature flow also arise in conformal field theory, where, according to Bakas \cite{Bakas2}, they describe ``to lowest order in perturbation theory ... the ultraviolet regime of the boundary renormalization group equation of Dirichlet sigma models''. 

Further interest in ancient and translating solutions to geometric flows arise from their rigidity properties, which are analogous to those of complete minimal/CMC surfaces, harmonic maps and Einstein metrics; for example, when $n\geq 2$, under certain geometric conditions, 
the only convex ancient solutions to mean curvature flow are the shrinking spheres \cite{HaHe,HuSi15,Wa11}. We present a new proof of this fact in section \ref{sec:shrinking sphere}. 


Convex ancient solutions to mean curvature flow are closely related to convex translating solutions. 
\begin{itemize}
\item[(i)] Consider a \emph{complete} solution $u:\Omega\to \R$, $\Omega\subset \R^{n}$, to the \textsc{pde}
\ba\label{eq:LSF/Trans}
\dvg\left(\frac{Du}{\sqrt{\sigma^2+\vert Du\vert^2}}\right)=-\frac{1}{\sqrt{\sigma^2+\vert Du\vert^2}}\,.
\ea
When $\sigma=0$, $u$ is the \emph{arrival time}\footnote{In this case, the equation is degenerate at critical points of $u$. At such points, $u$ satisfies \eqref{eq:LSF/Trans} in an appropriate \emph{viscosity} sense but is \emph{a priori} only continuous. When the level sets of $u$ are convex and compact, results of Huisken \cite{Hu84,Hu93} imply that $u$ has has a single critical point, where it is $C^2$.} of a mean convex ancient solution to mean curvature flow in $\R^{n}$ (i.e. the level sets of $u$ form a mean convex ancient solution). In this case, \eqref{eq:LSF/Trans} is called the \emph{level set flow}. When $\sigma=1$, the graph of $u$ is a mean convex translator in $\R^{n+1}$ with bulk velocity $-e_{n+1}$. In this case, \eqref{eq:LSF/Trans} is called the \emph{graphical translator equation}. The converse is also true: every mean convex ancient solution to mean curvature flow (resp. proper, mean convex translating solution) gives rise to a complete solution to \eqref{eq:LSF/Trans} with $\sigma=0$ (resp. $\sigma=1$). 
This connection was exploited by X.-J.~Wang in \cite{Wa11}. We will not make use of it here.
\item[(ii)] If $\{\M^n_t\}_{t\in(-\infty,\infty)}$ is a convex translating solution, then the family of rescaled solutions $\{\lambda\M^n_{\lambda^{-2}t}\}_{t\in(-\infty,0)}$ converges to a (self-similarly shrinking) convex ancient solution as $\lambda\to0$. See Lemma \ref{lem:translator_asymptotics} below.
\item[(iii)] If $\{\M^n_t\}_{t\in(-\infty,T)}$ is a convex ancient solution and $X_j\in \M^n_{t_j}$ are a sequence of points such that $t_j\to-\infty$ and $\vert X_j\vert\to\infty$ as $j\to\infty$, then a subsequence of the translated family $\{\M^n_{t+t_j}-X_j\}_{t\in(-\infty,T-t_j)}$ converges to a convex translating solution as $j\to\infty$. See Lemma \ref{lem:ancient_asymptotics} below.
\end{itemize}

\section{Asymptotics for convex ancient solutions}\label{sec:ancient_asymptotics}



The following lemma describes the asymptotic properties of convex ancient solutions. It will be convenient to make use of the inverse $P$ of the the Gauss map, which is defined on a strictly convex solution by
\[
\nu(P(e,t),t)=e\,.
\]

\begin{lemma}[Asymptotic shape of convex ancient solutions]\label{lem:ancient_asymptotics}
Suppose that $\{\M^n_t\}_{t\in(-\infty,0)}$ is a convex ancient solution to mean curvature flow in $\R^{n+1}$.
\begin{itemize}
\item[\emph{(i)}] \emph{(Blow-up)} 
\mbox{}
\begin{itemize}
\item[\emph{(a)}] If $\{\M^n_t\}_{t\in(-\infty,0)}$ is compact and $0$ is its singular time, there exists a point $p\in \R^{n+1}$ such that the family of rescaled solutions $\{\lambda(\M^n_{\lambda^{-2}t}-p)\}_{t\in(-\infty,0)}$ converges uniformly in the smooth topology to the shrinking sphere $\{S^n_{\sqrt{-2nt}}\}_{t\in(-\infty,0)}$ as $\lambda\to\infty$.
\item[\emph{(b)}] If $0$ is a regular time of $\{\M^n_t\}_{t\in(-\infty,0)}$ and $0\in \M^n_0$, then the family of rescaled solutions $\{\lambda\M^n_{\lambda^{-2}t}\}_{t\in(-\infty,0)}$ converges locally uniformly in the smooth topology to the stationary hyperplane $\{X:\inner{X}{\nu(0)}=0\}_{t\in(-\infty,0)}$ as $\lambda\to\infty$. 
\end{itemize}
\item[\emph{(ii)}] \emph{(Blow-down)} There is a rotation $R\in \mathrm{SO}(n+1)$ such that the rescaled solutions $\{\lambda R\cdot \M^n_{\lambda^{-2}t}\}_{t\in(-\infty,0)}$ converge locally uniformly in the smooth topology as $\lambda\to 0$ to either
\begin{itemize}
    \item the shrinking sphere $\{S^{n}_{\sqrt{-2nt}}\}_{t\in(-\infty,0)}$,
    \item the shrinking cylinder $\{\R^k\times S^{n-k}_{\sqrt{-2(n-k)t}}\}_{t\in(-\infty,0)}$ for some $k\in\{1,\dots,n-1\}$, or
    \item the stationary hyperplane $\{\R^n\times\{0\}\}_{t\in(-\infty,0)}$ of multiplicity either one or two.
\end{itemize}
Furthermore:\smallskip
\begin{itemize}
\item[\emph{(a)}] If the limit is the shrinking sphere, then $\{\M^n_t\}_{t\in(-\infty,0)}$ is a shrinking sphere.\smallskip 
\item[\emph{(b)}] If the limit is the stationary hyperplane of multiplicity one, then $\{\M^n_t\}_{t\in(-\infty,0)}$ is a stationary hyperplane of multiplicity one.
\end{itemize}
\item[\emph{(iii)}] \emph{(Asymptotic translators)} If the second fundamental form of the solution $\{\M^n_t\}_{t\in(-\infty,0)}$ is bounded on each time interval $(-\infty,T]$, $T<0$, then, given any sequence of times $s_j\to-\infty$ and a normal direction $e\in \cap_{s\in(-\infty,0)}\nu(\M^n_{s})$, the translated solutions $\{\M^n_{t+s_j}-P_j\}_{t\in(-\infty,-s_j)}$, where $P_j:=P(e,s_j)$, converge locally uniformly in $C^\infty$, along some subsequence, to a convex translating solution which translates in the direction $-e$ with speed $\displaystyle\lim_{j\to-\infty}H(P_j,s_j)$.
\end{itemize}
\end{lemma}
\begin{proof}
(i) The first part is an immediate consequence of Huisken's Theorem \cite{Hu84}. The second is straightforward.

\noindent (ii) This was proved by Wang through a delicate analysis of the arrival time \cite{Wa11}. We will derive it as a simple consequence of the monotonicity formula: after a finite translation, we can arrange that the solution reaches the origin at time zero. Since the shrinking sphere also reaches the origin at time zero, it must intersect the solution at all negative times by the avoidance principle. Thus,
\[
\min_{q\in \M^n_t}\vert q\vert\le \sqrt{-2nt}\,.
\]
Since the speed, $H$, of the solution is bounded in any compact subset of $\R^{n+1}\times(-\infty,0)$ after the rescaling, given any sequence $\lambda_i\searrow 0$, we can find a subsequence along which $\{\lambda_i\M^n_{\lambda_i^{-2}t}\}_{t\in(-\infty,0)}$ converges locally uniformly in the smooth topology to a non-empty limit flow.

We will show that the subsequential limit flow is a self-similarly shrinking solution using Huisken's monotonicity formula \cite{Hu90}:
\begin{equation}\label{eq:monotonicity formula}
\frac{d}{dt}\Theta(t)=-\int_{\M^n_t}\left\vert \vec H(p)+\frac{p^\perp}{-2t}\right\vert^2\Phi(p,t)\,d\mathscr{H}^n(p)\,,
\end{equation}
where $\Phi(\cdot,t)$ is the Gau\ss ian density and $\Theta(t)$ the Gaussian area of $\M^n_t$ (based at the spacetime origin):
\[
\Theta(t):=\int_{\M^n_t}\Phi(p,t)\,d\mathscr{H}^n(p)\,,\;\; \Phi(p,t)\doteqdot\left(-4\pi t\right)^{-\frac{n}{2}}\mathrm{e}^{-\frac{\left\vert p\right\vert^{2}}{-4t}}\,.
\]

We claim that $\Theta(t)$ is bounded uniformly.
\begin{claim}\label{ex:bdd_rescaled_volume}
There exists $C=C(n)<\infty$ such that
\[
\sup_{\tau>0}\frac{1}{\tau^{n/2}}\int_{\M^n} e^{\frac{-|y|^2}{\tau}}d\mathscr{H}^n(y)<C
\]
for any convex hypersurface $\M^n$ of $\R^{n+1}$ and any $\tau>0$.
\end{claim}
\begin{proof}
Let $A_i= B_{(i+1)\sqrt k}(0)\setminus B_{i\sqrt k}(0)$ for $i\in \N$. Since $\M^n$ is convex,
\[
|\M^n\cap A_i|\le |\partial B_{(i+1)\sqrt k}(0)|= c_n(i+1)^n k^{n/2}
\]
and hence
\[
\begin{split}
\frac{1}{\tau^{n/2}}\int_{\M^n} e^{\frac{-|y|^2}{\tau}}d\mathscr{H}^n(y)={}&\frac{1}{ \tau^{n/2}}\sum_{i=0}^\infty\int_{\M^n\cap A_i} e^{\frac{-|y|^2}{\tau}}d\mathscr{H}^n(y)\\
\le{}&\frac{1}{ \tau^{n/2}}\sum_{i=0}^\infty c_n(i+1)^n \tau^{n/2} e^{-i^2}\\
={}&c_n\sum_{i=0}^\infty(i+1)^n e^{-i^2}\,,
\end{split}
\]
where the constant $c_n$ depends only on $n$. This proves the claim.
\end{proof}

It follows that $\Theta$ converges to some limit as $t\to-\infty$. But then
\begin{align*}
-\int_a^b\!\!\!\int_{\lambda \M^n_{\lambda^{-2}t}}\left\vert\vec H(p)+\frac{p^\perp}{-2t}\right\vert^2\Phi(p,t)\,d\mathscr{H}^n(p)={}&\Theta_\lambda(b)-\Theta_\lambda(a)\\
={}&\Theta(\lambda^{-2}b)-\Theta(\lambda^{-2}a) \\
\to {}& 0
\end{align*}
as $\lambda\to 0$ for any $a<b<0$, where $\Theta_\lambda$ is the Gaussian area of the $\lambda$-rescaled flow. We conclude that the integrand vanishes identically in the limit and hence any limit of $\{\lambda \M^n_{\lambda^{-2}t}\}_{t\in(-\infty,0)}$ along a sequence of scales $\lambda_i\to 0$ is a self-similarly shrinking solution to mean curvature flow. 

A well-known result of Huisken \cite{Hu93} (see also Colding and Minicozzi \cite{ColdingMinicozziGeneric}) implies that the only convex examples which can arise are the shrinking spheres, the shrinking cylinders, and the stationary hyperplanes of multiplicity either one or two. If the limit flow is the shrinking sphere, then the solution is compact and hence, by part (i), its `blow up' is also the shrinking sphere. The monotonicity formula then implies that the Gaussian area is constant and we conclude that the solution is the shrinking sphere. Similarly, if the limit is the hyperplane of multiplicity one then the solution is the hyperplane of multiplicity one, since the blow-up about any regular point (without loss of generality, the spacetime origin) is also a hyperplane of multiplicity one (cf. \cite[Proposition 2.10]{White05}). Finally, we note that the limit is unique since convexity ensures that the limiting convex region enclosed by any subsequential limit is contained in the limiting convex region enclosed by any other subsequential limit.

\noindent (iii) This is a consequence of Hamilton's Harnack inequality \cite{HamiltonHarnack}. (We will in fact make use of Andrews' interpretation of the differential Harnack inequality using the Gauss map parametrization \cite{AndrewsHarnack}.) It implies that the mean curvature of a convex ancient solution to mean curvature flow is pointwise non-decreasing in the Gauss map parametrization. In particular, the limit 
\[
H_\infty(e):= \lim_{s\to-\infty}H(P(e,s),s)
\]
exists. Moreover, the curvature of the family of translated flows is uniformly bounded on any compact time interval. It follows that some subsequence converges locally uniformly in $C^\infty$ to a weakly convex eternal limit mean curvature flow. Since the curvature of the limit is constant in time with respect to the Gauss map parametrization, the rigidity case of the Harnack inequality \cite{HamiltonHarnack} implies that it moves by translation, with velocity $-H_\infty(e)e$.
\end{proof}

\begin{remark}
In case (iii) of Lemma \ref{lem:ancient_asymptotics}, the limit will simply be a stationary hyperplane if $\displaystyle\lim_{s\to-\infty}H(P(e,s),s)=0$. We refer to the `forwards' limit in part (i) as the \emph{blow-up} of the solution, the `backwards' limit in part (ii) as its \emph{blow-down} and the translating limits in part (iii) as \emph{asymptotic translators}. In certain cases, the asymptotic translator is unique, and hence the convergence holds for all $s\to-\infty$.
\end{remark}

\section{X.-J.~Wang's dichotomy for convex ancient solutions}\label{sec:Wang's dichotomy ancient}


Recall that the arrival time $u:\cup_{t\in I}\M^n_t\to\R$ of a mean curvature flow $\{\M^n_t\}_{t\in I}$ of mean convex boundaries $\M^n_t=\pd\Omega_t$ is defined by
\[
u(X)=t\;\;\iff\;\; X\in \M^n_t\,.
\]

\begin{lemma}\label{lem:concavity of the arrival time}
The arrival time of a convex ancient solution to mean curvature flow is locally concave.
\end{lemma}
\begin{proof}
Straightforward calculations show that the arrival time $u$ of a convex ancient solution $\{\M^n_t\}_{t\in(-\infty,T)}$ to mean curvature flow satisfies
\[
Du=-\frac{1}{H}\nu\;\;\text{and}\;\;D^2u=\left(\begin{matrix} -A/H & \cd H/H^2\\ \cd H/H^2 & -\pd_t H/H^3\end{matrix}\right)\,.
\]
Fix a point $p\in \M^n_t$ and any vector $V=V^\top+\alpha \nu(p,t)\in T_p\R^{n+1}$. If $\alpha=0$, then
\bann
-HD^2u(V,V)=A(V,V)\geq 0\,.
\eann
Else, we can scale $V$ so that $\alpha=-H$, in which case the differential Harnack inequality \cite{HamiltonHarnack} yields
\bann
-HD^2u(V,V)=A(V^\top,V^\top)+2\cd_{V^\top}H+\pd_t H\geq 0\,.
\eann
\end{proof}
In the compact case, concavity of the arrival time can also be obtained from Huisken's Theorem \cite{Hu84} and the concavity maximum principle \cite[Lemmas 4.1 and 4.4]{Wa11}.

\begin{definition}
An ancient mean curvature flow $\{\M^n_t\}_{t\in(-\infty,T)}$ of mean convex boundaries $\M^n_t=\pd\Omega_t$ shall be called \emph{entire}\index{entire!ancient solution}\index{ancient solution!entire} if it \emph{sweeps out all of space}, in the sense that $\cup_{t\in(-\infty,T)}\M^n_t=\R^{n+1}$. Equivalently, its arrival time is an entire function.
\end{definition}

Observe that the shrinking sphere and the bowl soliton are entire ancient solutions, whereas the paperclip and Grim Reaper are not entire (they only sweep out `strip' regions between two parallel lines).

Any convex ancient solution which is not entire necessarily lies at all times inside some stationary halfspace. The following remarkable theorem of X.-J.~Wang states that, in fact, such a solution must lie at all times in a stationary \emph{slab} region (the region between two parallel hyperplanes).\index{slab region}

\begin{theorem}[X.-J.~Wang's dichotomy for convex ancient solutions {\cite[Corollary 2.2]{Wa11}}]\label{thm:Wang's dichotomy for ancient solutions}
Let $\{\M^n_t\}_{t\in(-\infty,0)}$ be a convex ancient solution to mean curvature flow in $\R^{n+1}$. If $\{\M^n_t\}_{t\in(-\infty,0)}$ is not entire, then it lies in a stationary slab region\index{Wang's dichotomy!for ancient solutions}.
\end{theorem}
\begin{proof}
Let $\{\M^n_t\}_{t\in(-\infty,0)}$ be a convex ancient solution which is not entire. After a rotation, we can assume, by Lemma \ref{lem:ancient_asymptotics}, that the rescaled flows $\{\lambda \M^n_{\lambda^{-2}t}\}_{t\in(-\infty,0)}$ converge as $\lambda\to 0$ to the stationary hyperplane $L:= \{x\in \R^{n+1}:x_1=0\}$ with multiplicity $2$. Since the solution is convex, we may represent each timeslice $\M^n_t$ as the union of the graphs of two functions $v^{\pm}(\, \cdot \, ,t)$ over a convex domain $V_t\subset L$. That is,
\[
\M^n_t =\{(v^+(y,t),y): y\in V_t\}\cup\{(v^-(y,t), y): y\in V_t\}\,,
\]
with $v^+(\, \cdot \, ,t):V_t\to\R$ concave, $v^-:V_t\to\R$ convex, and $v^+$ lying above $v^-$. Thus, the function
\[
v:= v^+-v^-
\]
is positive and concave in its space variable at each time. Note also that $v(p,t)=0$ for each $p\in \pd V_t$. We need to show that $v(\, \cdot \, ,t)$ stays uniformly bounded as $t\to -\infty$. 

Since their graphs move by mean curvature flow, the functions $v^+$ and $v^-$ satisfy
\begin{equation}\label{eq:MCFvpm}
\frac{\partial v^\pm}{\partial t}=\sqrt{1+\vert Dv^\pm\vert^2}\dvg\left(\frac{Dv^\pm}{\sqrt{1+\vert Dv^\pm\vert^2}}\right)\,.
\end{equation}

Since the blow-down of the solution is the hyperplane $\{X\in\R^{n+1}:\inner{X}{e_1}=0\}$ of multiplicity two, there exists, for every $\varepsilon>0$, some $t_\varepsilon>-\infty$ such that
\begin{subequations}\label{eq:conv1}
\begin{equation}\label{eq:conv1a}
\vert p\vert\geq \varepsilon^{-1}\sqrt{-t}\;\;\text{for all}\;\; p\in \pd V_t
\end{equation}
and
\begin{equation}\label{eq:conv1b}
v(0,t)\leq \varepsilon\sqrt{-t}
\end{equation}
\end{subequations}
for all $t<t_\varepsilon$.

Since the solution is convex, it suffices, by \eqref{eq:conv1a}, to show that $v(0,t)$ stays bounded as $t\to -\infty$.
We achieve this with the following two claims.

\begin{claim}[{\cite[Claim 1 in Lemmas 2.1 and 2.2, and Lemma 2.6]{Wa11}}]\label{claim:rectangle}
There exists $t_0<0$ and $\alpha>0$ such that
\[
\vert p\vert\, v(0,t)\ge -\alpha t \,\, ( \, > 0 \, )
\]
for each $p\in\pd V_t$ and every $t\le t_0$.
\end{claim}
\begin{proof}
We will use equations \eqref{eq:conv1a} and \eqref{eq:conv1b} to show that the tangent planes to $v^+$ and $v^-$ at the origin are almost horizontal for $t\ll 0$, which will allow us to estimate $\vert p\vert \, v(0,t)$ by the area in the plane $e_1\wedge p$ enclosed by $\M_t$ and $\{\inner{X}{p}=0\}$. The argument is quite simple when $n=1$ and the solution is compact. Removing these hypotheses introduces some technical difficulties, but the idea is essentially the same.

Fix $t<t_\varepsilon$ and $p\in \pd V_t$. By rotating about the $x_1$-axis, we can arrange that $p=de_2$, where $d>0$. Set
\[
q^\pm(t):= v^\pm(0,t) \, e_1\,.
\]
By the convexity/concavity of the respective graphs, the segment connecting $(p,v^+(p))$ to $q^+(t)$ lies below the graph of $v^+(\, \cdot \, ,t)$ and the segment connecting $(p,v^-)$ to 
$q^-(t)$ lies above the graph of $v^-(\, \cdot \, ,t)$. Thus, comparing their slopes with the slope of the tangents to $v^\pm(\, \cdot \, ,t)$ at $0$ and applying \eqref{eq:conv1}, we find that
\begin{equation}\label{eq:tangent}
\pd_2v^+(0,t)-\pd_2v^-(0,t)\ge-\frac{v^+(0,t)-v^-(0,t)}{d}\ge -2\varepsilon^2\,.
\end{equation}
If the ray $\{re_2:r<0\}$ does not intersect $\pd V_t$, then, since $\M^n_t$ is convex,
\[
\pd_2v^+(0,t)-\pd_2v^-(0,t)\leq 0\,,
\]
Else, applying the same argument to the point $p':= \{re_2:r<0\}\cap\pd V_t$, we obtain
\[
\vert \pd_2v^+(0,t)-\pd_2v^-(0,t)\vert\le 2\varepsilon^2\,.
\]
Denote by $\widehat\M_t$ the intersection of $\M_t$ with the plane $e_1\wedge e_2$. Since the normal velocity of $\widehat\M_t$ in $e_1\wedge e_2$ is the projection onto $e_1\wedge e_2$ of $-H\nu$, the inward normal speed $\widehat H$ 
of $\widehat\M_t$
is given by
\[
\widehat H=H\frac{\vert D\widehat u\,\vert}{\vert Du\vert}\,,
\]
where $\widehat u$ is the restriction of the arrival time $u$ to $e_1\wedge e_2$. On the other hand, since $u$ is locally concave, the curvature $\widehat\kappa$ of 
$\widehat\M_t$ 
can be estimated by
\bann
\widehat\kappa={}&-\frac{1}{\vert D\widehat u\,\vert}\left(\Delta\widehat u-\frac{D^2\widehat u \, (D\widehat u,D\widehat u\,)}{\vert D\widehat u\,\vert^2}\right)\\
={}&-\frac{D^2\widehat u\, (\widehat\tau,\widehat\tau)}{\vert D\widehat u\,\vert}\\
={}&-\frac{1}{\vert D\widehat u\,\vert}\left(\Delta u-\frac{D^2u\, (Du,Du)}{\vert Du\vert^2}-\tr_{\widehat\tau{}^\perp}D^2u\right)\\
\leq{}&\frac{\vert Du\vert}{\vert D\widehat u\,\vert}H\,,
\eann
where $\widehat\tau$ is a unit tangent vector field to $\widehat\M_t$ and where $\widehat\tau{}^\perp$ is its orthogonal compliment in $T\M_t$. Thus,
\[
\widehat H\geq \widehat\kappa \, \frac{\vert D\widehat u\,\vert^2}{\vert Du\vert^2}\,.
\]

We claim that $\frac{\vert D\widehat u\,\vert^2}{\vert Du\vert^2}$ is close to $1$. Fix $x\in \widehat\M_t$. For each $k=3,\dots,n+1$, let $z_ke_k$, $z_k\in\R$, be the point at which the tangent plane $T_x\M_t$ intersects the $e_k$-axis (if it exists). Since $\inner{\nu}{z_ke_k-x}=0$, we find, when $\nu_k\neq 0$,
\[
z_k\nu_k=x_1\nu_1+x_2\nu_2\,.
\]
Without loss of generality, $x_1\leq x_2\leq z_k$ for each $k$. Then
\[
x_1^2+x_2^2\leq 2z_k^2
\]
and hence
\[
\nu_k^2\leq 2(\nu_1^2+\nu_2^2)
\]
for each $k=3,\dots,n+1$. Since $\vert\nu\vert^2=1$, this implies that
\[
\nu_1^2+\nu_2^2\geq\frac{1}{2n-1}
\]
and we conclude that
\[
\vert D\widehat u\,\vert^2=\inner{Du}{D\widehat u\,}\geq \frac{1}{2n-1}\vert Du\vert\vert D\widehat u\,\vert\,.
\]
That is,
\[
\frac{\vert D\widehat u\,\vert^2}{\vert Du\vert^2}\geq \frac{1}{(2n-1)^2}\,.
\]

We can now estimate
\bann
\frac{d}{dt}\mathscr{H}^2(\widehat\Omega_t\cap\{\inner{X}{e_2}>0\})={}&-\int_{\widehat\M_t\cap\{\inner{X}{e_2}>0\}}\widehat H\,ds\\
\le{}&-\tfrac{1}{(2n-1)^2}\int_{\widehat\M_t\cap\{\inner{X}{e_2}>0\}}\widehat \kappa\,ds\\
\le{}& -\tfrac{1}{(2n-1)^2}\left(\pi-2\varepsilon^2\right),
\eann
where $\widehat\Omega_t$ is the convex region in $e_1\wedge e_2$ bounded by $\widehat\M_t$. Integrating this between $t<4t_\varepsilon$ and $t_\varepsilon$ yields, upon choosing $\varepsilon=\frac{\sqrt{\pi}}{2}$,
\[
\mathscr{H}^2(\widehat\Omega_t^2\cap\{\inner{X}{e_2}>0\})\ge -\frac{3(\pi-2\varepsilon^2)}{4(2n-1)^2}\, t\ge -\frac{3\pi}{8(2n-1)^2}\, t\,.
\]
The convex region $\widehat\Omega_t$ lies between the tangent hyperplanes to $v^\pm(\, \cdot \, ,t)$ at the origin. By \eqref{eq:tangent}, these tangent hyperplanes intersect the line $\inner{X}{e_2}=d$ in $e_1\wedge e$ at two points with distance at most $2(v^+(0,t)-v^-(0,t))$. Comparing $\mathscr{H}^2(\widehat\Omega_t^2\cap\{\inner{X}{e_2}>0\})$ with the area of this enclosing 
trapezium 
yields
\[
\frac{3}{2}v(0,t) \, \vert p\vert\ge -\frac{3\pi}{8(2n-1)^2}t\,,
\]
which completes the proof of the claim.
\end{proof}

For each $k\in\N$, set
$$t_k:= 2^{k}t_\varepsilon \, , \quad
v_k:= v(\,\cdot\, ,t_k) \, , \quad \text{and} \quad d_k:= \min_{p\in \pd V_{t_k}}\vert p\vert \, ,$$
where $\varepsilon>0$ is to be determined and $t_\varepsilon$ is chosen as in \eqref{eq:conv1}.

\begin{claim}\cite[Claim 2 in Lemmas 2.1 and 2.2, and Lemma 2.7]{Wa11} \label{claim:thin}
There exists $\varepsilon>0$ such that
\begin{equation}\label{eq:thin} 
v_k(0)\le v_{k-1}(0)+2^{-\frac{k}{4n}}\sqrt{-t_\varepsilon}
\end{equation}
for all $k\geq 1$.
\end{claim}
\begin{proof}
Since the arrival time $u$ of $\{\M^n_t\}_{t\in(-\infty,0)}$ is locally concave and since
\[
u(v^\pm(y,t),y)=t
\]
for all $t<0$ and $y\in V_t$, we find that the functions $t\mapsto v^+(y,t)$ and $t\mapsto-v^-(y,t)$, and hence also $t\mapsto v(y,t)$, are concave for each fixed $y$. 

It follows that
\begin{equation}\label{vconcavity}
\frac{d}{dt}\frac{v(y,t)}{-t}\ge 0\;\; \text{for fixed}\;\; y\in V_{t_0}\,\, \text{and every}\,\, t<t_0\,.
\end{equation}
In particular,
\[
\frac{v_k(0)}{-t_k}\le \frac{v_0(0)}{-t_0}
\]
and hence, by \eqref{eq:conv1b},
\[
v_k(0)\le 2^{k}v_0(0)\le 2^{k+1}\varepsilon\sqrt{-t_\varepsilon}\,.
\]

Given $k_0\ge 8n$ (to be determined momentarily), we choose $\varepsilon=\varepsilon(k_0)$ small enough that $2^{k_0+1}\varepsilon\le 2^{-\tfrac {k_0}{4n}}$, so that
\begin{equation}\label{eq:gk0}
v_k(0)\le 2^{-\tfrac {k_0}{4n}}\sqrt{-t_\varepsilon}\le \frac{1}{4}\sqrt {-t_\varepsilon}\;\; \text{for all}\;\; k\le k_0.
\end{equation}
In particular, \eqref{eq:thin} holds for each $k\le k_0$. We will prove the claim by induction on $k$. 

So suppose that \eqref{eq:thin} holds up to some $k\ge k_0$. 
Then
\[
v_k(0)\le v_{k_0}(0)+\sqrt{-t_\varepsilon}\sum_{i=k_0+1}^k 2^{-\frac{i}{4n}}\,,
\]
where the second term on the right hand side is taken to be zero if $k=k_0$. Since $\sum_{j=1}^\infty 2^{-\frac{j}{4n}}<\infty$, we can choose $k_0$ so that $\sum_{j=k_0+1}^\infty 2^{-\frac{j}{4n}}<1/4$. Applying \eqref{eq:gk0}, we then obtain
\begin{equation}\label{eq:gk}
v_k(0)\le \frac12\sqrt{-t_\varepsilon}\,.
\end{equation}
By \eqref{vconcavity},
\[
\frac{v_{k+1}(0)}{-t_{k+1}}\le \frac{v_{k}(0)}{-t_k}
\]
and hence
\[
v_{k+1}(0)\le 2 v_k(0)\le \sqrt{-t_\varepsilon}\,.
\]
Since $t\mapsto v(0,t)$ is decreasing, we conclude that
\[
v(0,t)\le \sqrt{-t_\varepsilon}\;\;\text{for every}\;\; t\ge t_{k+1}\,.
\]

Since $d (t):= \min_{p\in \pd V_t}\vert p\vert$ is decreasing in $t$, Claim \ref{claim:rectangle} implies that there exists $\alpha > 0$ such that
\begin{subequations}\label{eq:ak+1}
\begin{equation}\label{eq:ak+1a}
d_{k+1}\ge d_k\ge \alpha\frac{-t_k} {v_{k}(0)}\ge\alpha\frac{-t_k}{\sqrt{-t_\varepsilon}}
\end{equation}
and
\begin{equation}\label{eq:ak+1b}
d_{k-1}\ge\alpha\frac{-t_{k-1}}{\sqrt{-t_\varepsilon}}=\frac{\alpha}{2}\frac{-t_k}{\sqrt{-t_\varepsilon}}\,.
\end{equation}
\end{subequations}

Define now 
\begin{equation}
D_k:= \left\{y\in\R^n:\vert y\vert<\frac{\alpha}{2}\frac{-t_k}{\sqrt{-t_\varepsilon}}\right\}\,.
\end{equation}
Since $y\mapsto v(y,t)$ is concave for fixed $t$,
\[
v(y+se,t)-v(y,t)\leq sD_ev(y,t)
\]
for any $y\in V_t$, any unit vector $e$ and any $s$ such that $y+se\in V_t$. If $(y,t)\in D_k\times [t_{k+1},t_{k}]$ and the ray $\{y+re:r>0\}$ intersects $\pd V_{t}$, we can choose $s$ so that $y+se\in\pd V_t$ and hence
\[
D_ev(y,t)\geq-\frac{v(y,t)}{s}\geq-\frac{v(y,t)}{d(t)-\vert y\vert}\,.
\]
If, however, the ray $\{y+re:r>0\}$ does not intersect $\pd V_{t}$, then, since $\M^n_t$ is convex, $D_{e}v(y,t)\geq 0$. Applying the same reasoning with $e$ replaced by $-e$ yields
\begin{equation*}
D_ev(y,t)\leq \frac{v(y,t)}{d(t)-\vert y\vert}\,.
\end{equation*}
We conclude the gradient estimate
\begin{equation}\label{eq:Wang's dichotomy claim 2 gradient estimate}
\vert Dv(y,t)\vert\leq \frac{v(y,t)}{d(t)-\vert y\vert}
\end{equation}
for all $(y,t)\in D_k\times[t_{k+1},t_k]$. 

Since $t\mapsto v(y,t)$ is concave for fixed $y$, \eqref{eq:Wang's dichotomy claim 2 gradient estimate} implies
\bann
v(y,t)\leq{}& v(0,t)+\vert y\vert \vert Dv(0,t)\vert\\
\leq{}& v(0,t)\left(1+\frac{\vert y\vert}{d(t)}\right)\\
\leq{}& 2v(0,t)\,.
\eann

These estimates, \eqref{eq:ak+1} and the concavity of $t\mapsto v(y,t)$ yield, for all $(y,t)\in D_k\times[t_{k+1},t_{k}]$,
\begin{equation}\label{eq:x1der}
|Dv(y,t)|\le \frac{v(y,t)}{d_k-|y|}\leq 
4\alpha^{-1}\frac{-t_\varepsilon}{-t_k}
\end{equation}
and
\begin{equation}\label{eq:hder}
0\le-\partial_t v(y,t)\le\frac{v(y,t)}{-t}\le\frac{v_k(y)}{-t_k}\le\frac{2\sqrt{-t_\varepsilon}}{-t_k}\,.
\end{equation}

We will use the gradient estimate to bound $\Delta v(y,t)$ in $D_k\times[t_{k+1},t_{k}]$ in terms of $t_k$. 
Given  a positive function $f:\N\to \R$, to be determined momentarily, define
\[
\chi_k:=\{(y,t)\in D_k\times[t_{k+1}, t_{k}]:-\Delta v(y,h)\ge f(k)\}\,.
\]
Recalling \eqref{eq:x1der} we find, for any $t\in (t_{k+1},t_{k})$,
\begin{align*}
\mathscr{H}^n(\{y\in D_k:(y, t)\in \chi_k\})\, f(k)\le{}&-\int_{D_k}\Delta v(\,\cdot\,,t)\,d\mathscr{H}^n\\
\leq{}&\int_{\pd D_k}\vert Dv\left(\,\cdot\,,t\right)\vert\,d\mathscr{H}^{n-1}\\
\le{}& \sup_{D_k}|Dv(\,\cdot\,,t)| \, \mathscr{H}^{n-1}(D_k)\\
\le{}&C\frac{t_\varepsilon}{t_k}\left(\frac{-t_k}{\sqrt{-t_\varepsilon}}\right)^{n-1}\\
={}&C\frac{2^{k(n-1)}}{-t_k}(-t_\varepsilon)^{\frac{n+1}{2}}\,,
\end{align*}
where $C$ is a constant which depends only on $n$. Integrating between $t_{k+1}$ and $t_k$ then yields
\[
\mathscr{H}^{n+1}(\chi_k)
\le C\frac{2^{k(n-1)}}{f(k)}(-t_\varepsilon)^{\frac{n+1}{2}}\cdot\frac{t_k-t_{k+1}}{-t_k}=2C\frac{2^{k(n-1)}}{f(k)}(-t_\varepsilon)^{\frac{n+1}{2}}\,.
\]

Consider now another positive function $g:\N\to\R$, to be determined later. Since
\begin{align*}
\int_{D_k}\mathscr{H}^1(\chi_k\cap\{ 
(y,t) : y=z 
\})\,d\mathscr{H}^n(z)
= & \, \mathscr{H}^{n+1}(\chi_k) \\
\le & \, 
2C\frac{2^{k(n-1)}}{f(k)}(-t_\varepsilon)^{\frac{n+1}{2}}\,,
\end{align*}
there exists $\widehat D_{k}\subset D_k$ with $$\mathscr{H}^n(\widehat D_{k})\le 2C\frac{2^{k(n-1)}}{f(k)g(k)}(-t_\varepsilon)^{\frac{n+1}{2}}$$ 
such that 
\begin{equation}\label{eq:Lsmall}
\mathscr{H}^1(\chi_k\cap\{ (y,t) : y=z \})\le g(k)\;\;\text{for all}\;\; z\in D_k\setminus \widehat D_{k}\,.
\end{equation}
Now, for any $y\in  D_k\setminus \widehat D_{k}$,
\begin{align}\label{Le}
v_{k+1}(y)-v_k(y)&=-\int_{t_{k+1}}^{t_{k}}\partial_tv(y,s)\,ds\nonumber\\
&=-\int_{[t_{k+1}, t_k]\setminus I_k(y)}\partial _tv(y,s)\,ds-\int_{I_k(y)}\partial _tv(y,s)\,ds\,,
\end{align}
where $I_k(y):=\{t:(y,t)\in \widehat D_k\}$. Using \eqref{eq:hder} and \eqref{eq:Lsmall} to bound the first integral on the right hand side of \eqref{Le} and the graphical mean curvature flow equation \eqref{eq:MCFvpm} to bound the second, we find, for any $y\in D_k\setminus \widehat D_k$,
\begin{equation}\label{Le2}
v_{k+1}(y)-v_k(y)\le \frac{2\sqrt{-t_\varepsilon}}{-t_k}g(k)-f(k)t_k\,.
\end{equation}

We now choose $f(k)= \frac{2^{-k(1+\beta/2)}}{\sqrt{-t_\varepsilon}}$ and $g(k)= 2^{k(1-\beta/2)}(- t_\varepsilon)$, where $\beta\in (0,1)$ will be chosen explicitly below. Then
\bann
\mathscr{H}^n(D_k\setminus \widehat D_k)\ge{}& \omega_n\left(\frac{-\pi t_k}{8\sqrt{-t_\varepsilon}}\right)^n-2^{k\beta+1}C \sqrt{-t_\varepsilon}\left(\frac{-t_k}{\sqrt{-t_\varepsilon}}\right)^{n-1}\\
={}&\left(\omega_n\left(\frac{\pi}{8}\right)^n2^k-2^{k\beta+1}C \right)\sqrt{-t_\varepsilon}\left(\frac{-t_k}{\sqrt{-t_\varepsilon}}\right)^{n-1}.
\eann

Since
\[
\mathscr{H}^n(\widehat D_k)\le 
2^{(\beta+n-1)k+1}C(-t_\varepsilon)^{\frac{n}{2}}\,,
\]
there is a point $y_0\in D_k\setminus \widehat D_k$ with 
\begin{equation}\label{eq: vk y0 le vk 0 exercise}
0<\vert y_0\vert \leq \left(\frac{2C}{\omega_n}\right)^{\frac{1}{n}}2^{k\left(\frac{\beta-1}{n}+1\right)}\sqrt{-t_\varepsilon} \quad \text{and} \quad v_k(y_0)\le v_k(0)\,.
\end{equation}
Since $y\mapsto v(y,t_{k+1})$ is concave and zero on $\pd V_t$, we find
\[
v_{k+1}(0)\le \frac {d_{k+1}}{d_{k+1}-\vert y\vert}v_{k+1}(y_0)\,.
\]

Set $C':= \left(\frac{2C}{\omega_n}\right)^{\frac{1}{n}}$. Then \eqref{eq:ak+1a} implies that
\bann
\frac{d_{k+1}}{d_{k+1}-\vert y\vert}\le{}&\frac{d_{k+1}}{d_{k+1}-2^{\left(\frac{\beta-1}{n}+1\right)k}C'\sqrt{-t_\varepsilon}}\\
={}&1+\frac {2^{\left(\frac{\beta-1}{n}+1\right)k}C'\sqrt{-t_\varepsilon}}{d_{k+1}-2^{\left(\frac{\beta-1}{n}+1\right)k}C'\sqrt{-t_\varepsilon}}\\
\le{}&1+2C'\cdot 2^{\frac{\beta-1}{n}k}
\eann
provided $k_0$ is chosen sufficiently large (depending on $\beta$). Thus,
\[
v_{k+1}(0)\le (1+2C'\cdot 2^{\frac{\beta-1}{n}k}) \, v_{k+1}(y_0)\,.
\]
Applying \eqref{Le2}, estimating $v_k(y_0)\le v_k(0)$ and recalling \eqref{eq:gk}, we conclude that
\[
\begin{split}
v_{k+1}(0)&\le(1+2C'\cdot 2^{\frac{\beta-1}{n}k})(v_{k}(0)+3\sqrt{-t_\varepsilon}\, 2^{-\frac{\beta}{2}k})\\
& \le v_k(0)+\sqrt{-t_\varepsilon}\left(C'\cdot 2^{\frac{\beta-1}{n}k}+3\cdot 2^{-\frac{\beta}{2}k}+6C'\cdot 2^{\left(\frac{\beta-1}{n}-\frac{\beta}{2}\right)k}\right)\,.
\end{split}
\]

Choosing now $\beta:=5/8\in (1/2n,3/4)$, we obtain, provided $k_0$ is sufficiently large,
\[
\begin{split}
v_{k+1}(0)\le{}&v_k(0)+\sqrt{-t_\varepsilon}
\, 2^{-\frac{k}{4n}}\,.
\end{split}
\]
This completes the proof of the claim.
\end{proof}
We conclude that
\[
v(0,2^{k+1}t_\varepsilon ) =
v_{k+1}(0)\leq v_0(0)+C\sum_{j=1}^k2^{-\frac{j}{4n}}\,,
\]
where $C$ is independent of $k$. Since the sum $\sum_{j=1}^\infty2^{-\frac{j}{4n}}$ is finite and $v(0,t)$ is monotone, we conclude that $v(0,t)$ is bounded uniformly in time. 

This completes the proof of Theorem \ref{thm:Wang's dichotomy for ancient solutions}.
\end{proof}

We note that F.~Chini and N.~M\o ller \cite{ChiniMollerAncient} have recently obtained a classification of the \emph{convex hulls} of (non-convex) ancient solutions to mean curvature flow more generally.

Theorem \ref{thm:Wang's dichotomy for ancient solutions} motivates the following definitions.
\begin{definition}
A convex, compact ancient solution to mean curvature flow is called an \emph{ancient ovaloid} if it is entire or an \emph{ancient pancake} if it lies in a stationary slab region.
\end{definition}

The Angenent oval is a well-known example of an ancient pancake, while the shrinking sphere is an example of an ancient ovaloid. Further ancient ovaloids were constructed by White \cite{Wh03}. Further ancient pancakes and ancient ovaloids were constructed by Wang \cite{Wa11}. Haslhofer and Hershkovits gave an explicit construction of a family of bi-symmetric ovaloids \cite{HaHe} and Angenent, Daskalopoulos and \v Se\v sum proved that there is only one (other than the shrinking sphere) which is uniformly two-convex and noncollapsing\footnote{In the sense of Sheng and Wang \cite{ShWa09} (see also \cite{An12,ALM13}).} \cite{ADS1,ADS2}. This is related to work of Brendle and Choi, who proved that the bowl is the only non-compact, strictly convex ancient solution which is uniformly two convex and noncollapsing \cite{BrendleChoi1,BrendleChoi2}. We gave an explicit construction of a rotationally symmetric ancient pancake and (by analyzing the asymptotic translators) proved that it is the only rotationally symmetric example \cite{BLT1}. A survey of these results can be found in \cite{BLTancient}.

\section{Convex ancient solutions to curve shortening flow}

Lemma \ref{lem:ancient_asymptotics} (ii) and Theorem \ref{thm:Wang's dichotomy for ancient solutions} immediately imply that the shrinking spheres are the only convex ancient solutions to curve shortening flow which do not lie in slab regions \cite{Wa11}. By making use of the differential Harnack inequality (as in Lemma \ref{lem:ancient_asymptotics} (iii)), an elementary enclosed area analysis, and Alexandrov's method of moving planes, we were able to obtain a complete classification of convex ancient solutions to curve shortening flow \cite{BLT3}.

\begin{theorem}\label{thm:CSF}
The only convex ancient solutions to curve shortening flow are the stationary lines, the shrinking circles, the Angenent ovals and the Grim Reapers.
\end{theorem}
The classification of \emph{compact} convex ancient solutions to curve shortening flow (shrinking circles and Angenent ovals) was obtained by Daskalopoulos, Hamilton and \v Se\v sum in 2010 \cite{DHS} using different methods.

Sections \ref{sec:shrinking sphere} and \ref{sec:bowl} illustrate how some of these ideas work in higher dimensions.

\section{Rigidity of the shrinking sphere}\label{sec:shrinking sphere}

We will use the asymptotics of the previous section to give a new proof of the following result of Huisken and Sinestrari \cite{HuSi15} (see also Haslhofer and Hershkovits \cite{HaHe} and Wang \cite[Remark 3.1]{Wa11}).

\begin{corollary}
\label{cor:sphererigidity}
Let $\{\M_t^n\}_{t\in(-\infty,0)}$, $n\geq 2$, be a compact, convex ancient solution to mean curvature flow. The following are equivalent:
\begin{enumerate}\item
\label{item:shrinkingsphere} $\{\M_t^n\}_{t\in(-\infty,0)}$ is a shrinking sphere $\{S^n_{\sqrt{-2nt}}(p)\}_{t\in(-\infty,0)}$, $p\in \R^{n+1}$.
\item
\label{item:uniformlypinched} $\{\M_t^n\}_{t\in(-\infty,0)}$ is uniformly pinched:
\[
\liminf_{t\to-\infty}\min_{\M^n_t}\frac{\kappa_1}{H}>0\,.
\]
\item
\label{item:brd} $\{\M_t^n\}_{t\in(-\infty,0)}$ has bounded rescaled \emph{(}extrinsic\emph{)} diameter:
\[
\limsup_{t\to-\infty}\frac{\diam(\M_t)}{\sqrt{-t}}<\infty\,.
\]
\item
\label{item:boundedeccentricity} $\{\M_t^n\}_{t\in(-\infty,0)}$ has bounded eccentricity: 
\[
\limsup_{t\to-\infty}\frac{\rho_+(t)}{\rho_-(t)}<\infty\,,
\]
where $\rho_+(t)$ and $\rho_-(t)$ denote, respectively, the circum- and in-radii of $\M_t^n$.
\item
\label{item:typeI} $\{\M_t^n\}_{t\in(-\infty,0)}$ has Type-I curvature decay:
\[
\limsup_{t\to-\infty}\sqrt{-t}\max_{\M^n_t}H<\infty\,.
\]
\end{enumerate}
\end{corollary}
\begin{proof}
Let $\{\M^n_t\}_{t\in(-\infty,0)}$ be a compact, convex ancient solution satisfying one of the conditions (\ref{item:uniformlypinched})--(\ref{item:typeI}). By Lemma \ref{lem:ancient_asymptotics} (ii-a), it suffices to prove that the blow-down of $\{\M^n_t\}_{t\in(-\infty,0)}$ is the shrinking sphere.

Condition \emph{\eqref{item:uniformlypinched}} (uniform pinching): Since the pinching condition is scale invariant and violated on any shrinking cylinder $\R^k\times S^{n-k}_{\sqrt{-2(n-k)t}}$, $k\in\{1,\dots,n-1\}$, it suffices, by Lemma \ref{lem:ancient_asymptotics} (ii), to rule out the hyperplane of multiplicity two. 

So suppose, to the contrary, that the blow-down is the hyperplane of multiplicity two. Theorem \ref{thm:Wang's dichotomy for ancient solutions} then implies that $\{\M^n_t\}_{t\in(-\infty,0)}$ lies in a slab. By translating parallel to the slab, we obtain, by part (iii) of Lemma \ref{lem:ancient_asymptotics}, an asymptotic translator $\Sigma^n$ which is non-trivial since it lies in a slab orthogonal to its translation vector. Thus, by the strong maximum principle, it must satisfy $H>0$. By the pinching condition, a theorem of Hamilton \cite{HamiltonPinched} then implies that $\Sigma^n$ is compact, which is impossible.


Condition \emph{\eqref{item:brd}} (bounded rescaled diameter) implies that the blow-down is compact. The theorem in this case then follows immediately from part (ii-a) of Lemma \ref{lem:ancient_asymptotics}.

Condition \emph{\eqref{item:boundedeccentricity}} (bounded eccentricity): This case is almost identical to case \emph{\ref{item:uniformlypinched}}. 

Condition \emph{\eqref{item:typeI}} (type-I curvature decay) implies (by integration) that the displacement is bounded by $\sqrt{-t}$. It follows that the blow-down is compact and the theorem again follows immediately from part (ii-a) of Lemma \ref{lem:ancient_asymptotics}.



\end{proof}

\section{Asymptotics for convex translators}



The following lemma describes the asymptotic geometry of convex translators.

\begin{lemma}[Asymptotic shape of convex translators]\label{lem:translator_asymptotics}
Let $\Sigma^n$ be a convex translator in $\R^{n+1}$ and $\{\Sigma_t^n\}_{t\in(-\infty,\infty)}$, where $\Sigma_t^n:= \Sigma^n+te_{n+1}$, the corresponding translating solution to mean curvature flow.
\begin{itemize}
\item[\emph{(i)}] \emph{(Asymptotic shrinker)} There is a rotation $R\in \mathrm{SO}(\R^n\times\{0\})$ about the $x_{n+1}$-axis such that the family of rescaled solutions $\{\lambda R\cdot \Sigma^n_{\lambda^{-2}t}\}_{t\in(-\infty,0)}$ converge locally uniformly in the smooth topology as $\lambda\to 0$ to either
\begin{itemize}
    \item the shrinking cylinder $\{S^{n-k}_{\sqrt{-2(n-k)t}}\times \R^k\}_{t\in(-\infty,0)}$ for some $k\in\{1,\dots,n-1\}$, or
    \item the stationary hyperplane $\{\{0\}\times\R^n\}_{t\in(-\infty,0)}$ of multiplicity either one or two.
\end{itemize}
Moreover, if the limit is the stationary hyperplane of multiplicity one, then $\Sigma^n$ is a vertical hyperplane.
\item[\emph{(ii)}] \emph{(Asymptotic translators\footnote{Cf. \cite[Lemma 2.1]{Ha15} and \cite[Lemma 3.1]{BL17}.})} Given any direction $e\in \pd\nu(\Sigma^n)$ in the boundary of the Gauss image $\nu(\Sigma^n)$ of $\Sigma^n$ and any sequence of points $X_j\in \Sigma^n$ with $\nu(X_j)\to e$, a subsequence of the translated solutions $\Sigma^n-X_j$ converges locally uniformly in the smooth topology to a convex translator which splits off a line (in the direction $w:= \lim_{j\to\infty}\frac{X_j}{\vert X_j\vert}$).
\end{itemize}
\end{lemma}
\begin{proof}
The first claim is a straightforward consequence of Lemma \ref{lem:ancient_asymptotics} (i). It was proved directly by Wang using a different argument \cite[Theorem 1.3]{Wa11}.

So consider the second claim. Up to a translation, we may assume that $X_1=0$. After passing to a subsequence, we can arrange that $w_j:= X_j/\Vert X_j\Vert\to w$ for some $w\in S^n$. Since each $\Sigma^n_j$ is convex and satisfies the translator equation, the sequence admits the uniform curvature bound
\[
\vert{A}_j\vert\leq H_j=-\inner{\nu_j}{e_{n+1}}\leq 1\,.
\]
It follows, after passing to a further subsequence, that the sequence converges locally uniformly in the smooth topology to a convex translator $\Sigma^n_\infty$. We claim that $\Sigma^n_\infty$ contains the line 
\[
L:= \{sw: s\in \R\}.
\]
First note that the closed convex region $\overline \Omega$ bounded by $\Sigma^n$ contains the ray $\{sw:s\geq 0\}$, since it contains each of the segments $\{sw_j:0\leq s\leq s_j\}$, where $s_j:= \Vert X_j\Vert$. By convexity, $\overline \Omega$ also contains the set $\{rsw+(1-r)X_j:s\geq 0, 0\leq r\leq 1\}$ for each $j$. It follows that the closed convex region $\overline\Omega_j$ bounded by $\Sigma^n_j$ contains the set $\{rsw-rs_jw_j:s>0, 0\leq r\leq 1\}$ since $s_jw_j = X_j$. In particular, choosing $s=2s_j$, $\{\vartheta (w-w_j)+\vartheta w:0\leq \vartheta\leq s_j\}\subset \overline\Omega_j$ and, choosing $s=s_j/2$, 
\[
\{\vartheta w_j-\vartheta (w-w_j):-s_j/2\leq \vartheta\leq 0\}\subset \overline\Omega_j \, .
\]
Taking $j\to\infty$, we find $\{sw:s\in\R\}\subset\overline \Omega_\infty$. It now follows from convexity of $\overline \Omega_\infty$ that $\{sw: s\in \R\}\subset \Sigma^n$. We conclude that $\kappa_1$ reaches zero somewhere on $\Sigma^n_\infty$ and the lemma follows from the splitting theorem.
\end{proof}


In some cases, the asymptotic translators can be shown to depend uniquely on $e\in\pd\nu(\Sigma^n)$, in which case the convergence is independent of the subsequence. 

Note that the asymptotic translator will simply be a vertical hyperplane if $e\perp e_{n+1}$.

\section{X.-J.~Wang's dichotomy for convex translators}



The following theorem is an immediate consequence of Theorem \ref{thm:Wang's dichotomy for ancient solutions}.

\begin{theorem}[X.-J.~Wang's dichotomy for convex translators {\cite[Corollary 2.2]{Wa11}}]\label{thm:Wang's dichotomy translators}
Let $\Sigma$ be a proper, convex translator. If $\Sigma$ is not entire, then it lies in a vertical slab region \emph{(}the region between two parallel vertical hyperplanes\emph{)}.
\end{theorem}

We note that F.~Chini and N.~M\o ller \cite{ChiniMollerTranslators} have recently obtained a classification of the \emph{convex hulls} of the projection onto the subspace orthogonal to the translation direction for general translators.

Theorem \ref{thm:Wang's dichotomy translators} motivates the following definitions.
\begin{definition}
A proper, locally uniformly convex translator is called a \emph{bowloid}\index{bowloid} if it is entire or a \emph{flying wing}\index{flying wing} if lies in a slab region.
\end{definition}


The Grim Reaper curve is a well-known example of a flying wing, while the bowl soliton is an example of a bowloid. Further flying wings and bowloids were constructed by Wang \cite{Wa11}, providing counterexamples to a conjecture of White \cite[Conjecture 2]{Wh03}. We proved that each slab of width at least $\pi$ in $\R^{n+1}$ admits a bi-symmetric\footnote{Namely, $\mathrm{O}(1)\times \mathrm{O}(n-1)$-symmetric.} flying wing. (Note that no slab of width less than $\pi$ admits a convex translator: the Grim hyperplane is a barrier.) Making use of the differential Harnack inequality (cf. Lemma \ref{lem:translator_asymptotics} (iii)), barrier arguments, and Alexandrov's method of moving planes, we also obtained unique asymptotics and reflection symmetry for such translators \cite{BLT2}. Hoffman, Ilmanen, Mart\'in and White constructed an $(n-1)$-parameter family of non-entire translating graphs and an $(n-2)$ parameter family of entire translating graphs in $\R^{n+1}$ for each $n\geq 2$ \cite{HIMW}. It is not yet clear whether or not these examples are convex, however (except when they coincide with the examples in \cite{BLT2}). A survey of these results can be found in \cite{BLTtranslators}.


\section{Rigidity of the bowl soliton}\label{sec:bowl}

We will use the asymptotics of Lemma \ref{lem:translator_asymptotics} in conjunction with the following important result of Haslhofer to obtain rigidity results for the bowl soliton.

\begin{theorem}[R.~Haslhofer \cite{Ha15}]\label{thm:Haslhofer asymptotic analysis}
Let $\Sigma^n$, $n\ge 2$, be a convex, locally uniformly convex translator in $\R^{n+1}$. Suppose that the blow-down of the corresponding translating solution $\{\Sigma^n_t\}_{t\in(-\infty,\infty)}$, $\Sigma_t:= \Sigma^n+te_{n+1}$, is the shrinking cylinder $S^{n-1}_{\sqrt{-2(n-1)t}}\times\R$. Then $\Sigma^n$ is rotationally symmetric about a vertial axis (and hence, up to a translation, the bowl soliton).
\end{theorem}
\begin{proof}
See \cite{Ha15}.
\end{proof}

\begin{corollary}[X.-J.~Wang {\cite[Theorem 1.1]{Wa11}}, J.~Spruck and L.~Xiao \cite{SX}]\label{cor:Wang bowl uniqueness}
Modulo translation, the bowl soliton is the only mean convex, entire translator $\Sigma^2$ in $\R^3$. In particular, it is the only translator which arises as a singularity model for a compact, embedded, mean convex mean curvature flow in $\R^3$.
\end{corollary}
\begin{proof}
By the Spruck--Xiao convexity estimate \cite{SX}, $\Sigma^2$ is actually convex. In fact, it must also be locally unformly convex: else, by the strong maximum principle, it would split off a line; by uniqueness of the Grim Reaper, the result would either be a vertical plane or a Grim plane, neither of which are entire. 

Since $\Sigma^2$ is entire, Lemma \ref{lem:translator_asymptotics} and Wang's dichotomy (Theorem \ref{thm:Wang's dichotomy translators}) imply that its blow-down is the shrinking cylinder, $\{S^{1}_{\sqrt{-2t}}\times\R\}_{t\in(-\infty,0)}$. The claim now follows from Theorem \ref{thm:Haslhofer asymptotic analysis}.
\end{proof}

Corollary \ref{cor:Wang bowl uniqueness} was proved by X.-J.~Wang (in the convex case) by a different argument: making use of the fact that the blow-down is the shrinking cylinder, he was able to obtain, by an iteration argument, the estimate
\[
\vert u(x)-u_0(x)\vert=o(\vert x\vert)
\]
as $\vert x\vert\to\infty$, where $\graph u_0$ is the bowl soliton whose tip coincides with that of $u$. A classical theorem of Bernstein then implies that $u-u_0$ is constant, which yields the claim. See \cite{Wa11}.

\begin{corollary}
\label{cor:bowl uniqueness}
Modulo translation, the bowl soliton is the only convex, non-planar translator $\Sigma^n$ in $\R^{n+1}$, $n\ge 3$, satisfying
\begin{equation}\label{eq:two convexity hypothesis}
\inf_{\Sigma^n}\frac{\kappa_1+\kappa_2}{H}>0\,.
\end{equation}
In particular, it is the only translator which arises as a singularity model for a compact, two-convex (immersed) mean curvature flow in $\R^{n+1}$ when $n\ge 3$.
\end{corollary}
\begin{proof}
First note that any convex, non-planar translator satisfying \eqref{eq:two convexity hypothesis} is automatically locally uniformly convex. Indeed, if this were not the case, then it would split off a line; but the complimentary subspace $\hat\Sigma^{n-1}$ cannot be compact (up to a rotation and scaling, it is a translator) and must satisfy $\inf_{\hat\Sigma^{n-1}}\kappa_1/H>0$, contradicting Hamilton's compactness theorem \cite{HamiltonPinched}.

We claim that $\Sigma^n$ is entire. Suppose, to the contrary, that this is not the case. Then, by Wang's dichotomy (Theorem \ref{thm:Wang's dichotomy translators}), it lies in a slab. Choose $e\in \pd\nu(\Sigma)$ with $\inner{e}{e_1}=0$. By Lemma \ref{lem:translator_asymptotics} (ii), we can find a sequence of points $X_j\in \Sigma^n$ such that the translates $\Sigma^n-X_j$ converge locally uniformly in the smooth topology to a convex translator which splits off a line parallel to the slab and has normal $e$ at the origin. Since $\inner{e_1}{e}=0$, the limit cannot be a vertical hyperplane. It follows that the limit has positive mean curvature. Moreover, since $\inner{e}{e_{n+1}}<0$, the cross section of the splitting cannot be compact. But this contradicts the two-convexity hypothesis \eqref{eq:two convexity hypothesis} by Hamilton's compactness theorem \cite{HamiltonPinched}. The claim now follows from  Haslhofer's analysis (Theorem \ref{thm:Haslhofer asymptotic analysis}).
\end{proof}

To our knowledge, Corollary \ref{cor:bowl uniqueness} was only previously known under additional hypotheses, such as noncollapsing \cite{Ha15} or certain cylindrical and gradient estimates \cite{BL17}.

\bibliographystyle{acm}
\bibliography{../../bibliography}

\end{document}